\documentclass[12pt, a4paper, reqno]{amsart}
\usepackage{a4, amsfonts, amsthm, verbatim,times, mathrsfs, pstricks,pst-plot, hyperref, fullpage, enumitem}
\usepackage{amssymb,latexsym,amsmath}

\usepackage[english]{babel}
\usepackage[T1]{fontenc}

\usepackage{xcolor %, refcheck
 }
\newcommand{\SO}{\mathrm{SO}}
\newcommand{\HH}{{\mathbb H}}
\newcommand{\RR}{{\mathbb R}}
\newcommand{\R}{{\mathbb R}}

 \newcommand{\ZZ}{{\mathbb Z}}
\newcommand{\wi}{\widetilde}

\renewcommand{\phi}{\varphi}

\newcommand{\cercle}{\mathbb{S}}

\newcommand{\Ric}{\text{Ric}}
\newcommand{\spin}{\mathrm{Spin}}
\newcommand{\C}{\mathbb{C}}
\newcommand{\Spin}{\mathrm{Spin}}

\newcommand{\End}{\mathrm{End}}
\newcommand{\Spinc}{\text{Spin}^c}
\newcommand{\spinc}{\Spinc}
\newcommand{\id}{\mathrm{Id}} 
 
\let\<\langle
\let\>\rangle
\renewcommand{\d}{\mathrm{d}}
\renewcommand{\i}{\mathrm{i}}

\newtheorem{thm}{Theorem}[section]
\newtheorem{lemma}[thm]{Lemma}
\newtheorem{prop}[thm]{Proposition}
\newtheorem{cor}[thm]{Corollary}

\theoremstyle{definition}      
\newtheorem{remark}[thm]{Remark}

\begin{document}

\title{Complex Generalized Killing Spinors on  Riemannian Spin$^c$ Manifolds}

\author{Nadine Gro{\ss}e}
\address{Mathematisches Institut, 
Universit\"at Leipzig, 
04009 Leipzig, 
Germany}
\email{grosse@math.uni-leipzig.de}

\author{Roger Nakad}
\address{Department of Mathematics and Statistics, Faculty of Natural and Applied Sciences, Notre Dame University-Louaize, P.O. Box 72, Zouk Mikael, Zouk Mosbeh, Lebanon.}
\email{rnakad@ndu.edu.lb}

\begin{abstract} 
In this paper, we extend the study of  generalized Killing spinors on  Riemannian  $\Spinc$ manifolds started by Moroianu and Herzlich to complex Killing functions. We prove that such spinor fields  are always  real  $\spinc$ Killing spinors or imaginary generalized $\Spinc$ Killing spinors, providing that the dimension of the manifold  is greater or equal to $4$. Moreover, we classify Riemannian $\Spinc$ manifolds carrying imaginary and imaginary generalized Killing spinors.
\end{abstract}

\maketitle

\section{Introduction}
 
On a Riemannian $\spin$  manifold $(M^n, g)$ of dimension $n \geq 2$, a non-trivial spinor field $\psi$ is called a {\it complex generalized Killing spinor field}  with smooth Killing function $K$ if
\begin{eqnarray}\label{CGKSspin}
\nabla_X\psi = K \,X\cdot\psi,
\end{eqnarray}
for all vector fields $X$ on $M$, where $\nabla$ denotes the spinorial Levi-Civita connection and $``\cdot"$ the Clifford multiplication. Here $K := a +\i b$ denotes a complex function with real part function $a$ and imaginary part function $b$.\\

It is well known that the existence of such spinors imposes several restrictions on the geometry and the topology of the manifold. More  precisely, on a Riemannian $\spin$ manifold, a {\it complex generalized Killing spinor} is either a {\it real generalized Killing spinor} (i.e., $b = 0$ and $a \neq 0$),  an {\it imaginary generalized Killing spinor} (i.e., $a = 0$ and $b\neq 0$) or  a {\it parallel spinor} (i.e., $b=a = 0$) \cite{bfgk, CGLP}. Manifolds with parallel spinor fields are Ricci-flat and can be characterised by their holonomy group \cite{hitchin, frr}. Riemannian $\spin$ manifolds carrying parallel  spinors  have been classified by M. Wang \cite{wang}.\\

When $\psi $ is a {\it real generalized Killing spinor}, then $a$ is already a nonzero constant, i.e., $\psi$ is in fact a real Killing spinor. Those Killing spinors on simply connected Riemannian $\spin$ manifolds were classified by C. B\"ar  \cite{bar}. Real Killing
spinors occur in physics, e.g. in supergravity theories, see \cite{DN}, but they are also of
mathematical interest: The existence of real Killing spinor field implies that the manifold is a compact Einstein manifold of scalar curvature  $4n(n-1)a^2$. In dimension 4,  it has constant sectional curvature. Real Killing spinors are also special solutions of the twistor
equation \cite{lich, lich2} and moreover, they are related to the spectrum of the Dirac operator. In fact,  T. Friedrich \cite{fri} proved a lower
bound for the eigenvalues of the Dirac operator involving the infimum of the scalar
curvature. The equality case is characterised by the existence of a real Killing spinor. More precisely, $n^2 a^2$ is the smallest  
eigenvalue of the square of the Dirac operator \cite{fri, OH2}. Other geometric and physics applications of the existence of  real Killing spinors can be found in  \cite{CGLP, franc, NP, Su, OH1, OH2, TF1, TF2, TF3, TK1, TK2, TK3}.  When $\psi$ is an  imaginary generalized  Killing spinor, then $M$ is a non-compact Einstein manifold. Moreover,  two cases may occur: The function $b$ could be  constant (then $\psi$ is called an {\it imaginary Killing spinor}) or it is a non-constant function (then we will continue to call $\psi$ a {\it imaginary generalized Killing spinor}). H. Baum \cite{baum2, baum, baum1} classified Riemannian $\spin$ manifolds carrying imaginary Killing spinors.  Shortly later,  H-B. Rademacher  extended this classification to imaginary generalized Killing spinors  \cite{rada} -- here only so called type I Killing spinors can occur, see Section \ref{sec_im}.\\ 

Recently, $\Spinc$ geometry  became a field of active research with the advent of Seiberg-Witten theory. Applications of the Seiberg-Witten theory to $4$-dimensional geometry and topology are already notorious. From an intrinsic point of view,  $\spin$,  almost complex, complex, K\"{a}hler, Sasaki and some classes of CR manifolds have a canonical $\Spinc$ structure. Having a $\spinc$ structure is a weaker condition than having a $\spin$ structure. Moreover, when shifting from the classical $\spin$ geometry to $\Spinc$ geometry, the situation is more general  since the connection on the $\spinc$ bundle, its curvature, the Dirac operator and its spectrum  will not only depend on the geometry of the manifold but also on the connection (and hence the curvature) of
 the auxiliary line bundle associated with the  $\Spinc$ structure.\\

A. Moroianu studied  Equation~\eqref{CGKSspin} on Riemannian $\Spinc$ manifolds when $b = 0$ and $a$ is constant, i.e., when $\psi$ is a parallel spinor or a real Killing spinor \cite{moro}. He proved that a simply connected complete Riemannian  $\Spinc$ manifold carrying a parallel spinor is isometric to the Riemannian product of a K\"ahler manifold (endowed with its canonical $\Spinc$ structure) with a $\spin$ manifold carrying a parallel spinor. Moreover, a simply connected complete Riemannian  $\Spinc$ manifold carrying a real Killing spinor is isometric to a Sasakian manifold endowed with its canonical $\Spinc$ structure.  In 1999, M. Herzlich and A. Moroianu considered Equation~\eqref{CGKSspin} for $b = 0$ on Riemannian $\Spinc$ manifolds \cite{8}. They proved that, if $n \geq 4$,  real generalized  $\Spinc$ Killing 
spinor do not exist, i.e., they are already  real $\spinc$ Killing spinor. In dimension $2$ and $3$, they constructed explicit examples of   $\Spinc$ manifolds carrying real generalized Killing spinor, i.e., where the real  Killing function is not constant. \\

We recall also that the existence of parallel spinors, real Killing spinors and imaginary Killing spinors do not only give obstruction of the geometry and the topology of the $\spin$ or $\Spinc$  manifold $(M^n, g)$ itself, but also, the geometry and topology of hypersurfaces and submanifolds of $(M^n, g)$. In fact, The restriction of such $\spin$ or $\Spinc$ spinors  is an effective tool to  study the geometry and the topology of submanifolds \cite{amm_00, amm_05, bar_98, h1, h2, h3, h4, h5, n1, n2}.\\

In this paper, we extend the study of  Equation~\eqref{CGKSspin} on Riemannian $\Spinc$ manifolds. After giving some preliminaries of $\Spin^c$ structures in Section \ref{prelim}, we consider general properties of complex generalized Killing spinors in Section \ref{sec_compl} and prove

\begin{thm}\label{no_complex} 
Let $(M^n, g)$ be a connected Riemannian $\Spinc$ manifold of dimension $n \geq 4$, carrying a complex  generalized Killing spinor $\psi$ with Killing function $K = a + \i b$, $a,b\in C^\infty(M,\R)$. Then, $a$ or $b$ vanishes identically on $M$. In other words, $\psi$ is already a real generalized Killing spinor with Killing function $a$ (and hence $a$ is constant, i.e, $\psi$ is a real Killing spinor) or $\psi$ is an imaginary generalized Killing spinor with Killing function $ib$ ($b$ is constant or a function). 
\end{thm}

Proof of Theorem~\ref{no_complex} is based on the existence of  differential forms which are  naturally
associated to the complex generalized Killing spinor field. For $n=2,3$ we still do not know whether there are complex generalized Killing spinors which are neither purely real or purely imaginary. But at least we know that $a$ has to vanish in all points where $b$ does not and vice versa, cf. Lemma~\ref{ab0}. Thus, in case they exist they would be very artificial, cf. Remark \ref{rem_low}. In dimension $\geq 4$ these cases are excluded since even locally there are no non-constant real Killing functions.\\

 Since parallel and real generalized Killing spinors were already studied in \cite{moro, 8}, we can focus  then on studying \eqref{CGKSspin} with $a =   0$, i.e., we give a classification of Riemannian $\spinc$ manifolds carrying imaginary generalized Killing spinors. Most of the cases appearing then will be the obvious generalization of the $\spin$ case, see Theorem~\ref{type_I} and Proposition \ref{type_II_con}. But in contrast to the $\spin$ case, type II imaginary generalized Killing spinors with non-constant Killing function exist -- but only in dimension $2$, cf. Proposition \ref{no_gen_II_1} and Theorem~\ref{gen_II_2}.

%%%%%%%%%%%%%%%%%%%%%%%%%%%%%%%%%%%%%%%%%%%%%%%%%%%
\section{Preliminaries} \label{prelim}
%%%%%%%%%%%%%%%%%%%%%%%%%%%%%%%%%%%%%%%%%%%%%%%%%%%

\subsection{Conventions and general notations.}
Hermitian products $\<.,.\>$ are always anti-linear in the second component. If a vector is decorated with a hat, this vector is left out, e.g. $v_1,\ldots, \hat{v}_j,\ldots, v_n$ is meant to be $v_1,\ldots, v_{j-1}, v_{j+1},\ldots, v_n$. The space of smooth sections of a bundle $E$ is denoted by $\Gamma (E)$.

\subsection{Spin$^c$ structures on manifolds.}\label{prelim_spinc}

We consider an oriented Riemannian manifold  $(M^n, g)$  of dimension $n\geq 2$  without boundary and denote by $\SO (M)$ the 
$\SO_n$-principal bundle over $M$ of positively oriented orthonormal frames. A $\Spinc$ structure of $M$ is given by an $\cercle^1$-principal bundle $(\cercle^1 M ,\pi,M)$ of some Hermitian line bundle $L$ and a $\Spin_n^c$-principal bundle $(\Spinc M,\pi,M)$ which is a $2$-fold covering of the $\SO_n\times\cercle^1$-principal bundle $\SO (M)\times_{M}\cercle^1 M$ compatible with the group covering
$$0 \longrightarrow \ZZ_2 \longrightarrow \Spin_n^c = \Spin_n\times_{\ZZ_2}\cercle^1 \longrightarrow \SO_n\times\cercle^1 \longrightarrow 0.$$
The bundle $L$ is called the auxiliary line bundle associated with the $\Spinc$ structure. If $A: T(\cercle^1 M)\longrightarrow \i\RR$ is 
a connection 1-form on $\cercle^1 M$, its (imaginary-valued) curvature will be denoted by $F_A$, whereas we shall define a real 
$2$-form $\Omega$ on $\cercle^1 M$ by $F_A= \i\Omega$. We know
that $\Omega$ can be viewed as a real valued 2-form on $M$ \cite{6, koba1}. In this case, $\i\Omega$ is the curvature form of the auxiliary line bundle $L$ \cite{6, koba1}.\\

Let $\Sigma M := \Spinc M \times_{\rho_n} \Sigma_n$ be the associated spinor bundle where $\Sigma_n = \C^{2^{[\frac n2]}}$ and $\rho_n : \Spin_n^c
\longrightarrow  \End(\Sigma_{n})$ the complex spinor representation \cite{6, LaMi}. 
A section of $\Sigma M$ will be called a spinor field. This complex vector bundle is naturally endowed
 with a Clifford multiplication, denoted by ``$\cdot$'', $\cdot : \C l(TM) \longrightarrow \End(\Sigma M)$
 which is a fiber preserving algebra morphism, and with a natural Hermitian scalar product $\< . , .\>$ 
compatible with this Clifford multiplication \cite{6, hijazi4}. If such data are given, one can canonically define a covariant derivative 
$\nabla$ on $\Sigma M$ that is locally given by \cite{6, hijazi4, 2ana}:
\begin{eqnarray}\label{nnnabla}
 \nabla_X \psi = X(\psi) + \frac 14 \sum_{j=1}^n e_j\cdot \nabla_X^M e_j\cdot\psi + \frac \i2 A(s_*(X))\psi,
\end{eqnarray}
where $X \in \Gamma(TM)$, $\nabla^M$ is the Levi-Civita connection on $M$, $\psi = [\widetilde{b\times s}, \sigma]$ is a locally defined spinor field, $b=(e_1, \ldots, e_n)$ is a local oriented orthonormal tangent frame over an open set $U\subset M$, 
$s : U \longrightarrow \cercle^1M|_U$ is a local section of $\cercle^1M$,  $\widetilde{b\times s}$ is the lift of the local section $b\times s :
 U \longrightarrow \SO (M) \times_M \cercle^1 M|_U$ to the $2$-fold covering and
  $X(\psi) = [\widetilde{b\times s}, X(\sigma)]$. \\

The Dirac operator, acting on $\Gamma(\Sigma M)$, is a first order elliptic operator locally given by  $D =\sum_{j=1}^n e_j \cdot \nabla_{e_j},$
where $\{e_j\}_{j=1, \ldots, n}$ is any local orthonormal frame on $M$. An important tool when examining the Dirac operator on $\Spinc$ manifolds is the Schr\"{o}dinger-Lichnerowicz formula \cite{6}:
\begin{eqnarray}
D^2 = \nabla^*\nabla + \frac 14 S\; \id_{\Gamma (\Sigma M)}+ \frac{\i}{2}\Omega\cdot,
\label{SL}
\end{eqnarray}
where $S$ is the scalar curvature of $M$, $\nabla^*$ is the adjoint of $\nabla$ with respect to the $L^2$-scalar product and $\cdot$ is the extension of the Clifford multiplication to differential forms. 
The Ricci identity is given, for all $X\in \Gamma(TM)$,  by 
\begin{eqnarray*}
\sum_{j=1}^n e_j\cdot \mathcal R(e_j, X)\psi = \frac 12 \Ric(X)\cdot\psi - \frac \i2 (X\lrcorner \Omega)\cdot\psi,
\end{eqnarray*}
for any spinor field $\psi$. Here, $\mathrm{\Ric}$ (resp. $\mathcal R$) denotes the Ricci tensor 
of $M$ (resp. the $\Spinc$ curvature associated with the connection $\nabla$), and $\lrcorner$ is the interior product.\\

Let $\omega_{\mathbb{C}}=\i^{[\frac{n+1}{2}]} e_1\wedge \ldots \wedge e_n$ be the complexified volume element. The Clifford multiplication extends to differential forms so $\omega_{\mathbb C}$  can act on spinors. If $n$ is odd, the volume element $\omega_{\mathbb C}$ acts as the identity on the spinor bundle.
If $n$ is even,  $\omega_{\mathbb C}^2=1$. Thus, by the action of the complex volume element on the spinor bundle  decomposes into the  eigenspaces $\Sigma^{\pm} M$ corresponding to the $\pm 1$ eigenspaces, the {\it positive} (resp. {\it negative}) spinors
\cite{6, hijazi4, 2ana}. If $\psi=\psi_+ + \psi_-$ for $\psi_{\pm}\in \Gamma(\Sigma^{\pm} M)$, we set $\bar{\phi}=\psi_+ - \psi_-$.  Summarizing the action of the volume form we have
\begin{align}\label{vol_form}
\omega_{\mathbb C}\cdot \psi=\left\{\begin{matrix} \bar{\psi} & \text{for\ } n \text{\ even}\\
\psi & \text{for\ } n \text{\ odd.}
\end{matrix}
\right.
\end{align}
Moreover, we recall that by direct calculation one sees immediately that 
\begin{equation}\label{prod_im_re} \<\delta\cdot \psi, \psi\> = (-1)^{\frac{k(k+1)}{2}} \overline {\<\delta\cdot \psi, \psi\>}\end{equation} for a $k$-form $\delta$ and a spinor field $\psi$.\\ 

Furthermore, it is well known that a $\spin$ structure can be viewed as a $\spinc$ structure with a trivial auxiliary line bundle endowed with the trivial connection. Of course, $\Spinc$ manifolds are not in general $\spin$ manifolds -- e.g. the complex projective space $\mathbb C P^2$ is $\spinc$ but not $\spin$.  However, a $\Spinc$ structure on a simply
connected Riemannian manifold $M$ with trivial auxiliary bundle $L$ is canonically identified with a $\spin$ structure. Moreover, if
 the connection defined on the trivial auxiliary line bundle is flat, then  $\nabla$ on the $\spinc$ bundle $\Sigma M$ corresponds to $\nabla^{'}$
on the $\Spin$ bundle $\Sigma^{'}M$, i.e., we have a global section on $L$ which can be chosen parallel  \cite[Lemma $2.1$]{moro}. In this case, \eqref{nnnabla} becomes
\begin{eqnarray}\label{nnabla}
\nabla'_X \psi =  \nabla_X \psi = X(\psi) + \frac 14 \sum_{j=1}^n e_j\cdot \nabla^M_Xe_j\cdot\psi. 
\end{eqnarray}

Let $(M^n, g)$ be a Riemannian $\Spinc$ manifold with $\Spinc$ bundle $\Sigma M$ and auxiliary bundle $L$. Let now $(M^n,g)$ be equipped with another $\Spinc$ structure, and let $\Sigma' M$ (resp. $L'$)  the corresponding $\Spinc$ bundle (resp. auxiliary bundle). Then there is always a complex line bundle $\mathcal{D}$ such that $\Sigma^{'}M = \Sigma M \otimes {\mathcal D}$ and $L^{'} = L \otimes {\mathcal D}^2$.
In particular, if $(M,g)$ is $\spin$ and $\Sigma'M$ denotes its spinor bundle and  $L'$ the trivial line bundle. Then, ${\mathcal D}^2 = L^{-1}$. Thus $\Sigma M = \Sigma^{'} M \otimes L^{\frac 12}$. Even if $M$ is not $\spin$ this is still true locally. This essentially means that, while the spinor bundle and $L^{\frac 12}$ may not exist globally, their tensor product (the $\Spinc$ bundle) can be defined globally.  

\subsection {Conformal Killing vector fields}\label{conf-Kill} We denote by $\mathcal L_Vg$ the Lie derivative of the metric $g$ in direction of the
vector field $V$. A vector field $V$ is a {\it conformal
Killing field} if 
$\mathcal L_Vg = 2hg$  for a smooth real function h. By taking traces one obtains $ \mathrm{div} V=nh$. $V$ is {\it homothetic} if $h$ is a constant, and it is {\it isometric} if $\mathcal L_Vg = 0$. Moreover, $V$ is called {\it closed} if the corresponding $1$-form $w = g(V, .)$ is closed. H.-B. Rademacher proved that 

\begin{thm}\cite[Theorem~2]{rada}\label{rada_thm}
Let $(M^n, g)$ be a complete Riemannian manifold with a non-isometric conformal closed Killing vector field $V$, and let $N$ be the number of zeros of $V$. Then $N \geq 2$  and: 
\begin{enumerate}
\item If $N=2$, $M$ is conformally diffeomorphic to the standard sphere $\mathbb S^n$.
\item If $N=1$, $M$ is conformally diffeomorphic to the Euclidean space $\mathbb R^n$.
\item If $N=0$,  there exists a complete $(n - 1)$-dimensional Riemannian manifold $(F, g_F)$ and a smooth function $h: \mathbb R \longrightarrow \mathbb R^*$ such that the warped product $F \times_h \mathbb R$  is a
Riemannian covering of $M$, and the lift of $V$ is $h \frac{\partial}{\partial t}$ where $t$ denotes the coordinate of the $\mathbb R$-factor.
\end{enumerate}
\end{thm}

%%%%%%%%%%%%%%%%%%%%%%%%%%%%%%%%%%%%%%%%%%%%%%%%%%%%%%%%%%%%%%%%%%%%%%%%%%
\section {Spin$^c$ Complex Generalized Killing spinors}\label{sec_compl}

In this section, we want to establish general properties of  complex generalized Killing spinors, i.e. a spinor satisfying \eqref{CGKSspin}. In particular, we will show that in dimension $n\geq 4$, the Killing function $K$ is already purely real or purely imaginary. First, let us collect some general facts on Killing spinors.

\begin{lemma}\label{gen}
Let $\psi$ be  a complex generalized Killing  spinor on a Riemannian $\Spinc$ manifold $(M^n, g)$. Let $K=a+\i b$ be the corresponding Killing function. Then
\begin{enumerate}[label=(\rm{\roman{*}}), ref=\rm{\roman{*}}]
\item\label{Lich_compl} $n(n-1) K^2 \psi -(n -1) \d K\cdot  \psi = \frac 14 S \psi + \frac \i2 \Omega \cdot\psi$
\item\label{dabspin} $\<\d a\cdot \psi, \psi\> = 2n ab\i |\psi|^2$
\item\label{ric_kill}
$ \frac{1}{2}(\Ric (X)- \i X \lrcorner \Omega )\cdot \psi = \nabla K\cdot X\cdot \psi + nX(K)\psi +2(n-1)K^2 X\cdot \psi$
\item\label{gen_pos} $\psi$ has no zeros.
\end{enumerate}
\end{lemma}

\begin{proof} All the calculations will be carried out at a point $x\in M$ using a local orthonormal frame  $e_i$ with $[e_i,e_j]=0$ and $\nabla_{e_i} e_j=0$ at $x$. We calculate $D \psi = \sum_i e_i\cdot \nabla_{e_i}\psi = K \sum_i e_i\cdot e_i \cdot \psi = -nK\psi$. Thus, 
\begin{eqnarray*}
D^2 \psi = -n \d K + n^2 K^2 \psi.
\end{eqnarray*}
Moreover, using $ \nabla^* \nabla  = - \sum_{j=1}^n \nabla_{e_j} \nabla_{e_j}$, we deduce that 
\begin{eqnarray*}
\nabla^* \nabla \psi = -\d K\cdot\psi  + n K^2 \psi.
\end{eqnarray*}
Using the last two equations and the Schr\" odinger Lichnerowicz formula \eqref{SL}, we obtain \eqref{Lich_compl}.
Then, taking the imaginary part of the scalar product of \eqref{Lich_compl} $\psi$, we get that  $\<da\cdot\psi, \psi\> = 2n ab\i |\psi|^2$.
For the last two claims the corresponding proofs for $K$ real, i.e., 
\cite[Lemma~2.2]{8} for \eqref{ric_kill}   and \cite[Proposition~1]{lich} for \eqref{gen_pos}, carry over directly.
\end{proof}

Let $\omega_p$ be the $p$-form on $M$ given by 
$$w_p (X_1, \dots, X_p) := \<(X_1\wedge X_2 \wedge \dots \wedge X_p)\cdot\psi, \psi\>,$$
for any $X_1, X_2, \dots, X_p \in \Gamma(TM)$. These $p$-forms have been introduced in \cite{8} for  real generalized Killing $\Spinc$ spinors.  In this case, the vector field $V := \i\omega_1^{\flat}$ is a Killing vector field. We point out that for complex generalized Killing $\Spinc$ spinors, this is not the case, i.e., $\xi$ is not necessary a Killing vector field, cp. Section \ref{sec_im}.

%%%%%%%%%%%%%%%%%%%%%%%%%%%%
%%%%%%%%%%%%%%%%%%%%%%%%%%%%
\begin{lemma}\label{p-form-first}
The forms  $\omega_{4k+1}$ and $\omega_{4k+2}$ are imaginary-valued, but $\omega_{4k+3}$ and $\omega_{4k}$ are real-valued forms for all $k\geq 0$. Moreover,  we have for all $p\geq 0$
$$\d\omega_p = (K (-1)^p - \overline K) \omega_{p+1}.$$
In particular, for any  $k \geq 0$, 
\begin{eqnarray}\label{eq1}
\d b \wedge \omega_{2k+1} - 2ab \omega_{2k+2} &=0,\\
\label{eq2}
\d a \wedge \omega_{2k+2} + 2ab\i \omega_{2k+3} &= 0.
\end{eqnarray}
\end{lemma}

\begin{proof} By \eqref{prod_im_re}, $\omega_{4k+1}$ and $\omega_{4k+2}$ are imaginary-valued; $\omega_{4k+3}$ and $\omega_{4k+4}$ are real-valued forms.  We consider a local orthonormal frame $\{e_1, \dots, e_n\}$ in a neighbourhood of $x\in M$ such that   $[e_i ,e_j] =0$ and  $\nabla_{e_i} {e_j} =0$ at $x$. Then
\begin{align*}
(&p+1) \d \omega_p (e_1,\ldots ,e_{p+1})\\
&= \sum_{i=1} ^{p+1} (-1)^{i-1} e_i \Big(\omega_p (e_1,...,\hat{e}_i,...,e_{p+1}) \Big) + \sum_{i < j} (-1)^{i+j} \omega_p \Big([e_i ,e_j ],e_1,...,\hat{e}_i,...,\widehat {e_j},...,e_{p+1}\Big)\\ 
%&= \sum_{i=1} ^{p+1} (-1)^{i-1}  e_i \Big(\<e_1 \cdot \ldots \cdot\hat{e}_i \cdot \ldots\cdot  e_{p+1} \cdot \psi,\psi\>\Big )\\ 
&=\sum_{i=1} ^{p+1} (-1)^{i-1}  \Big[\<\nabla_{e_i} (e_1 \cdot \ldots\hat{e}_i  \cdot \ldots \cdot e_{p+1}\cdot\psi),\psi\>+ \<e_1 \cdot \ldots\cdot\hat{e}_i\cdot \ldots\cdot e_{p+1}\cdot\psi,\nabla_{e_i }\psi\>\Big]\\
&=\sum_{i=1} ^{p+1} (-1)^{i-1}  \Big[\<e_1 \cdot \ldots \cdot \hat{e}_i \cdot \ldots\cdot e_{p+1}\cdot\nabla_{e_i} \psi,\psi\> +\< e_1 \cdot \ldots \cdot \hat{e}_i \cdot \ldots \cdot e_{p+1}\cdot\psi, K e_i \cdot\psi\>\Big]\\
&=\sum_{i=1} ^{p+1} (-1)^{i-1} \Big[ \< e_1 \cdot ... \cdot \hat{e}_i\cdot ... \cdot e_{p+1} \cdot (K e_i \cdot\psi),\psi\>  - \overline K \< e_i \cdot\Big (e_1 \cdot ... \cdot\hat{e}_i\cdot ... \cdot e_{p+1}\Big)\cdot\psi,\psi\>\Big]\\
 &= (p+1)(K(-1)^p - \overline K) \omega_{p+1}.
\end{align*}
 Thus, we get
\begin{equation*}
 \left\{
\begin{array}{rcl}
\d\omega_{2k}&=&2\i b \omega_{2k+1},\\
\d\omega_{2k+1}&=&-2 a \omega_{2k+2}.
\end{array}\right.
\end{equation*}
After taking the differential of the last two equalities, we obtain $\d b \wedge \omega_{2k+1} - 2ab \omega_{2k+2} =0$ and  $\d a \wedge \omega_{2k+2} + 2ab \i \omega_{2k+3} = 0$.
\end{proof}

If $M$ is even dimensional, we can use the decomposition of the spinor bundle, see  \eqref{vol_form} and above, to define another sequence of $p$-forms on $M$ by 
$$\overline\omega_p (X_1, \ldots, X_p) := \<X_1\cdot X_2 \cdot \ldots \cdot X_p\cdot\psi, \overline\psi\>,$$
for $X_1, X_2, \dots, X_p \in \Gamma(TM)$.

\begin{lemma}
If $n$ is even, the $p$-form $\overline \omega_p$ satisfies 
$$\d\overline \omega_p = (K (-1)^p + \overline K) \overline\omega_{p+1}.$$
In particular, for any  $k \geq 0$, 
\begin{eqnarray*}
\d b \wedge \overline\omega_{2k+2} + 2ab \overline\omega_{2k+3} =0,
\end{eqnarray*}
\begin{eqnarray}\label{eq2-bar}
\d a \wedge \overline\omega_{2k+1} - 2ab\i \overline\omega_{2k+2} = 0.
\end{eqnarray}
\end{lemma}
\begin{proof} For $X\in \Gamma(TM)$ the Clifford multiplication $X\cdot$ is a map from $\Gamma(\Sigma^\pm M)$ to $\Gamma(\Sigma^\mp M)$. Thus, $\nabla_X\overline \psi = - K X \cdot \overline \psi$.  Now we can proceed as in Lemma~\ref{p-form-first} and obtain 
\begin{eqnarray*}
\d \overline\omega_p (e_1,\ldots ,e_{p+1}) &= 
 (K (-1)^p + \overline K)\overline \omega_{p+1}.
\end{eqnarray*}
Thus, we get
\begin{equation*}
 \left\{
\begin{array}{rcl}
\d\overline\omega_{2k}&=&2a \overline\omega_{2k+1},\\
\d\overline\omega_{2k+1}&=&-2 \i b \overline\omega_{2k+2}.
\end{array}\right.
\end{equation*}
Taking the differential, we obtain $\d b \wedge \overline\omega_{2k+2} + 2ab \overline\omega_{2k+3} =0$ and  $\d a\wedge \overline\omega_{2k+1} - 2ab\i \overline\omega_{2k+2} = 0$.
\end{proof}

\begin{lemma} \label{ab0}
Let $(M^n, g)$ be a Riemannian $\spinc$ manifold carrying  a complex generalized Killing spinor with Killing function $K= a+\i b$. Then $ab=0$. 
\end{lemma}

\begin{proof} Let $\psi$ denote the Killing spinor, and let ${e_1, e_2, \ldots, e_n}$ be a local orthonormal frame of $TM$. Firstly, assume that $n$ is odd and set $k = \frac{n-3}{2}$.  Equality~\eqref{eq2} for $k$ implies that
$$\d a \wedge \omega_{n-1} = -2ab \i \omega_n.$$ 
We calculate each term of this equation separately. First, we have  
\begin{align*}
(\d a \wedge \omega_{n-1}) (e_1, e_2,\ldots, e_n) &= \sum_{j=1}^n (-1)^{j+1} \d a(e_j) \omega_{n-1} (e_1, e_2,\ldots, \hat e_j, \ldots, e_n)\\
&=\sum_{j=1}^n  (-1)^{j+1} \d a(e_j) \<e_1\cdot e_2 \ldots\cdot\hat e_j \cdot \ldots\cdot e_n \cdot\psi, \psi\>.
\end{align*} 
Using \eqref{vol_form}, we get $(-1)^j \i^{[\frac{n+1}{2}]} e_1\cdot e_2\cdot\ldots\cdot \hat e_j\cdot \ldots\cdot e_n\cdot\psi = e_j \cdot\psi.$
Thus, we have 
\begin{align*}
(\d a \wedge \omega_{n-1}) (e_1, e_2,\ldots, e_n) &=  \sum_{j=1}^n  (-1)^{j+1} (-1)^{-j} \i^{-[\frac{n+1}{2}]} \d a(e_j) \<e_j \cdot\psi, \psi\> 
\\
&= - \i^{-[\frac{n+1}{2}]}  \sum_{j=1}^n \d a(e_j) \<e_j \cdot\psi, \psi\>  
\\
&=   - \i^{-[\frac{n+1}{2}]} \< \d a\cdot\psi, \psi\>.
\end{align*} 
On the other hand
\begin{align*}
-2ab \i \omega_n(e_1, e_2, \ldots, e_n) &= -2ab\i \<e_1\cdot e_2 \cdot\ldots\cdot e_n \cdot \psi, \psi\>  = -2ab \i \i^{-[\frac{n+1}{2}]} |\psi|^2.
\end{align*}
Thus, $$-2ab \i \i^{-[\frac{n+1}{2}]} |\psi|^2 =  - \i^{-[\frac{n+1}{2}]} \<\d a\cdot\psi, \psi\>.$$
Together with Lemma~\ref{gen}\eqref{dabspin} and~\ref{gen}\eqref{gen_pos}, we obtain that  $2ab \i   = 2n ab\i$. Hence, $ab = 0$. 

It remains the case that $n$ is even.  Then, \eqref{eq2-bar} for $k = \frac{n-2}{2}$ implies $\d a \wedge \overline \omega_{n-1} = 2ab \i \overline \omega_n$,
and an analogous calculation as in the first case gives again $ab = 0$.
\end{proof}
Now, we are able to prove Theorem~\ref{no_complex}. 
\begin{proof} [Proof of Theorem~\ref{no_complex}.] We prove the claim by contradiction, i.e. let $\psi$ be a Killing spinor to a Killing function $a+\i b$ where not both $a$ and $b$ vanish identically.
Set $\Omega:=\{x\in M\ |\ b(x)=0\}$. Then,   $\psi|_{\Omega}$ is a real generalized Killing spinor to the Killing function $a|_{\Omega}\not\equiv 0$. For $n\geq 4$, this implies that  $a$ has to be constant on $\Omega$ \cite[Theorem~1.1]{8}. But by Lemma~\ref{ab0},  we know that $ab = 0$. Thus, $a|_{M\setminus \Omega}=0$ which gives a contradiction to the smoothness of $a$.
\end{proof}

\begin{remark}\label{rem_low} We conjecture that complex generalized Killing $\Spinc$ spinors also do not exist in dimension $2$ and $3$. Even, if this turns out to be wrong, these examples are very artificial:  The manifold $M$ consists of two closed subsets $M_1$ and $M_2$ where $\psi|_{M_1}$ is a real generalized Killing spinor on $M_1$ to the Killing function $a$ and $\psi|_{M_2}$ is a real generalized Killing spinor on $M_1$ to the Killing function $\i b$. In particular, on $M_1\cap M_2$, we have $a=b=0$ and everything has to built such that it is smooth also over this ``boundary'' set.  For the imaginary spinor part on $M_2$,  this is clearly possible when taking e.g. a warped product as in Theorem~\ref{type_I} by choosing $k(t)$ carefully. But whether one can choose the real part such that the spinor has a good well-behaved zero set is still unclear.
\end{remark}

%%%%%%%%%%%%%%%%%%%%%%%%%%%%%%%%%%%%%%%%%%%%%%%%%%%%%%%%%%%%%%%%%%%%%%%%%%%

\section{Spin$^c$ Imaginary Generalized Killing  spinors}\label{sec_im}
 
On a Riemannian $\Spinc$ manifold $(M^n, g)$, we consider 
a imaginary generalized Killing spinor $\psi$ with Killing function $\i b$, where $b$ is a smooth real function that is not identically zero on $M$. Let $f := \vert\psi\vert^2$. Moreover, define the vector field $V$  by 
\begin{eqnarray}\label{def_V} g(V, X)  =  \i \< X\cdot\psi, \psi \>\quad \text{\ for\ all\ } X \in \Gamma(TM).\end{eqnarray}
As in the $\Spin$ case we get by direct computation, \cite[Section 3]{rada}
\begin{eqnarray}\label{conf_Kill}
\nabla f = 2b V, \quad \nabla_X V =2 b f X, \quad \mathcal {L}_V g = 4b f g, 
\end{eqnarray}
for all $X\in \Gamma(TM)$. Hence, the vector field $V$ is a non-isometric conformal closed Killing vector field, \cite[Section 2]{rada} and cf. Paragraph \ref{conf-Kill}. Moreover, the function $q_\psi := f^2 - \Vert V\Vert^2$ is non-negative constant  and \begin{eqnarray} \label{q0}            
\frac{1}{f} V\cdot\psi =- \i\psi \quad \text{\ for\ }q_\psi = 0.
\end{eqnarray}
The proof of this follows exactly the one  in the $\spin$ case \cite[Lemma~5 and below]{baum}.
 The spinor field $\psi$  is called of \emph{type} 
I (resp. II) if $q_\psi = 0$ (resp. $q_\psi > 0$).

\subsection{Imaginary Generalized Killing  spinors of type I}\label{type_I_sec}

We start with the type \rm{I} imaginary generalized Killing spinors. It turns out that one only obtains the obvious generalization of the corresponding $\Spin$ result \cite[Theorem~1a]{rada}.

\begin{thm}\label{type_I}
Let $(M^n, g)$ be a complete connected Riemannian $\spinc$ manifold admitting a imaginary generalized Killing spinor of
 type {\rm I} with Killing function $\i b$, $b\in C^\infty(M,\mathbb{R})$. Then, a Riemannian covering of $M$ is isometric to the warped product 
$ F\times_k \RR = (F^{n-1}  \times \RR, k(t)^2 h \oplus \d t^2 ),$ where $(F^{n-1}, h)$ is a complete Riemannian 
$\spin^c$ manifold admitting a non-zero parallel spinor field, and $k$ is a function on $t$. In particular, $f(t, x) = k(t)$ is also a function on $t$ alone and $b = \frac{f^{'}}{2f}$.  Moreover, every manifold that fulfills these conditions admits a imaginary generalized Killing spinor of type {\rm I}.
\end{thm}

\begin{proof} The proof is analogous to the ordinary $\Spin$ case: For a type \rm{I} imaginary  Killing spinor $\psi$, $q_\psi=0$ and hence, $\Vert V\Vert =f=|\psi|^2$. Then by Lemma~\ref{gen}\eqref{gen_pos} $V$ has no zeros. By Theorem~\ref{rada_thm}, a Riemannian covering of $M$ is the warped product of a complete Riemannian manifold $(F^{n-1}, h)$ and $(\RR, \d t^2$), 
warped by a positive smooth function $k(t)$, i.e. $(F\times_k \RR, k^2(t)h+ \d t^2)$. The lift of $V$ to this covering is given by $k\frac{\partial}{\partial t}$. Then, $k(t)=\Vert V\Vert =f$.
Thus, using  \eqref{conf_Kill}, we get 
$$f^{'} \frac{\partial}{\partial t}= \nabla f = 2b V = 2b f\frac{\partial}{\partial t}.$$ 
Hence, $b = \frac{f^{'}}{2 f}$. Moreover, the manifold $F_t := F\times_f \{t\} = (F, f(t)h)$ can be viewed as a hypersurface of $M$ whose mean curvature  with respect to the unit normal vector field $\partial_t = \frac 1f V$ is given by $-\frac{f^{'}}{f}$ \cite[Example~4.2]{bgm}. Hence, $F := F_0$ carries an induced $\spin^c$ structure \cite{2ana}. 
Using \eqref{q0} and the $\spin^c$ Gauss formula \cite[Proposition~3.3]{2ana}, 
we calculate for $\phi= \psi\vert_{F}$
\begin{eqnarray*}
\nabla_X^{F}\phi = \left(\nabla_X\psi\right)\vert_{F}  +  \frac{f^{'}}{2f} X\cdot\partial_t\cdot\psi\vert_F = \i b X\cdot\psi\vert_F - 
\i b X\cdot\psi\vert_F =0,
\end{eqnarray*}
where $\nabla^F$ is the $\spinc$ connection on $F$. This gives a parallel spinor field $\phi$ on $F$. 

For the converse, let $\phi$ be a nonzero parallel spinor on $(F^{n-1}, h)$. By parallel transport of $\phi$ in $t$-direction we get $\phi(t,x)$. Firstly  assume that $n$ is odd, i.e., $n= 2m +1$. Then, we can assume that w.l.o.g. that $\phi$ is in one of the $S^{\pm}_F$ such that $\partial_t\cdot \phi= (-1)^m \i  \phi$ where $\phi$ is now seen as a spinor in $S_M$, cp. \cite[Lemma~4]{baum}.
Set $\psi(t,x)= \eta(t) \phi(t,x)$ with $\eta(t)=e^{-\int_0^t\frac{k'(s)}{2k(s)} \d s}$ and $b=(-1)^m \frac{k'}{2k}$. Then for $X\in \Gamma(TM)$ with $X\perp \partial_t$ we get 

\begin{eqnarray*}
\nabla_X\psi = \eta \nabla^{F_t}_X\phi -  \eta \frac{k^{'}}{2k} X\cdot\partial_t\cdot\phi = \i b X\cdot\psi.
\end{eqnarray*}
Moreover, $\nabla_{\partial_t}\psi = \eta' \phi  = -\i(-1)^m\eta' \partial_t\cdot \phi= \i(-1)^m\frac{k'}{2k}\eta \partial_t\cdot \phi=\i b\partial_t\cdot \psi $. Thus, $\psi$ is a Killing spinor to Killing function $b$. Moreover,  $\Vert V\Vert= |g(V,\partial_t)|=|i\<\partial_t\cdot \psi, \psi\>|= |\psi|^2=f $, thus, $\psi$ is of type I.
Similar we obtain the Killing spinor when $n$ is even: As in \cite[Lemma~4]{baum} $\tilde{\phi}=\phi\oplus \phi$ can be seen as a spinor in $S_M$ with $\partial_t\cdot \tilde{\phi}=(-1)^m \i \tilde{\phi}$ and $n=2m+2$. 

Set $b=(-1)^m \frac{k'}{2k}$. Then for $X\in \Gamma(TM)$ with $X\perp \partial_t$ we get 
%\begin{eqnarray*}
$\nabla_X\psi =  -  \eta \frac{k^{'}}{2k} X\cdot\partial_t\cdot\phi = \i b X\cdot\psi$
%\end{eqnarray*}
and $\nabla_{\partial_t}\psi = \eta' \phi  = \i b \partial_t\cdot\psi$. Thus, $\psi$ is a Killing spinor to Killing function $b$. Moreover,  $\Vert V\Vert= |g(V,\partial_t)|=|\i\<\partial_t\cdot \psi, \psi\>|= |\psi|^2=f $, thus, $\psi$ is of type I.
\end{proof}

\begin{cor}
Let $(M^n, g)$ be a complete connected Riemannian $\spinc$ manifold admitting an imaginary Killing spinor of Killing number $i\mu$, $\mu\in \RR$. If $\psi$ is of type {\rm I}, a Riemannian covering of $M$ is isometric to the warped product 
$(F^{n-1}  \times \RR, e^{4\mu t} h \oplus dt^2 ),$ where $(F^{n-1}, h)$ is a complete $\spin^c$ manifold with a non-zero parallel spinor. 
\end{cor}

\begin{proof} By Theorem~\ref{type_I} it only remains to determine $f$. As above we have, $f'=2\mu f$. Thus, $f=ae^{2\mu t}$ for a positive constant $a$. By rescaling the metric $h$, we can assume that $a =1$.  
\end{proof}

\subsection{Imaginary Generalized Killing spinor of type II}\label{type_II_sec}

Next we study type \rm{II} generalized  imaginary Killing spinors to the Killing function $\i b$.  We will distinguish two cases: 

\subsubsection{\bf $b$ is constant.}

Then it turns out that $M$ is already $\Spin$:

\begin{prop}\label{type_II_con}
Let $(M^n, g)$ be a complete connected Riemannian $\spinc$ manifold with an imaginary Killing spinor $\psi$ of Killing number $\i\mu$, $\mu\in \RR\setminus \{0\}$. If $\psi$ is of type {\rm{II}}, $(M^n, g)$ is 
isometric to the hyperbolic space $\HH^n(-4\mu^2)$ endowed with its trivial $\spinc$ structure, i.e., its unique $\spin$ structure.
\end{prop}

\begin{proof} Let $\psi$ be of type \rm{II}, i.e., $q_\psi >0$. First we assume that $f=|\psi|^2$ has no critical points, then, by \eqref{conf_Kill} the number of zeros of $V$ is $0$. From Theorem~\ref{rada_thm} we obtain that a Riemannian covering of $M$ is isometric to the warped product $F \times_{k}\RR$ where $k(t)$ is a function on $t$ alone, $F$ a complete Riemannian manifold and the lift of $V$ is $k\frac{\partial}{\partial t}$. 
Then again with \eqref{conf_Kill} we obtain that $f$ also just depends on $t$ and $f'\frac{\partial}{\partial t}=\nabla f= 2\mu V$. Thus, $f'= 2\mu k$ and $f''\frac{\partial}{\partial t}=2\mu \nabla_{\partial_t} V=4\mu^2f \frac{\partial}{\partial t}$.  Hence, $f=Ae^{2\mu t} +B e^{-2\mu t}$ for constants $A,B$. Since $V$ and, hence, $f'$ has no zeros,   $f'=2\mu k$ and since $k$ and $f$ are everywhere positive, we obtain $f=Ae^{2\mu t}$, $A>0$, and $k=f$. Hence, $k= \Vert V\Vert=f$ and $q_\psi =0$ which gives a contradiction. 

Hence, $f$ has  critical points. Using \eqref{conf_Kill} we obtain for $X, Y \in \Gamma(TM)$
 \[\mathrm{Hess} f (X, Y) = g(\nabla_X \nabla f, Y)= 2\mu g(\nabla_X V,Y)= 4 \mu^2 g(X, Y) f.\]

 By \cite[Theorem~C]{kanai}, $M$ is isometric to the simply connected complete Riemannian manifold $(\HH^{n}, (2|\mu|)^{-1}g_{\HH})$ of constant curvature $-4\mu^2$. Since $\mathbb H^n$ is contractible,  $\mathbb H^n$ admits only one $\Spinc$ structure -- the canonical one coming from the $\Spin$ structure. 

By Lemma~\ref{gen}\eqref{ric_kill} we obtain 
 $$\mathrm{\Ric}(X)\cdot \psi - i(X\lrcorner \Omega)\cdot\psi = -4(n-1)\mu^2 X\cdot\psi$$
for all $X\in \Gamma(TM)$. Since the Ricci tensor of $M$ is given by ${\Ric}= -4(n-1)\mu^2$, we obtain $(X\lrcorner \Omega)\cdot\psi =0$ for all $X \in \Gamma(TM)$ and, hence, $\i\Omega=0$. 

Thus, the $\spinc$ structure is identified with the unique  $\spin$ structure on $\HH^n$. Here we recall that, 
on $\HH^n(-4\mu^2)$ endowed with its unique $\spin$ structure, imaginary Killing spinors of Killing number $\i\mu$ and $-\i\mu$ form an orthogonal basis of $\HH^n$ with respect to the Hermitian scalar product defined on $\Sigma M$ \cite{baum}.
\end{proof}

\subsubsection{\bf $b$ is not constant.}
On $\spin$ manifolds, H.-B. Rademacher proved that there are no imaginary generalized Killing spinors of type \rm{II} where $b$ is non-constant. For dimension $n\geq 3$, this will be still true for $\Spinc$ manifolds. In contrast, in dimension $2$ such spinors exist. In order to carry out the case of generalized $\Spinc$ Killing spinors, we need the following auxiliary lemma.

\begin{lemma}\label{aux_lem} Let $\psi$ be a generalized Killing spinor to the Killing function $\i b$, $b\in C^\infty (M, \RR)$. Then, in all points of $M$  where $\nabla b\neq 0$ 
\[ \< X\cdot \psi, \psi \>=0\quad \text{\ for\ all\ } X\perp \nabla b.\]
\end{lemma}

\begin{proof}
From \eqref{eq1}, we have that $\d b \wedge \omega_1 = 0$. Let $X\perp \nabla b$. Then
$$0 = (\d b \wedge \omega_1) (\nabla b, X) = \vert \nabla b \vert^2 \omega_1 (X) =\vert \nabla b \vert^2  \<X \cdot\psi, \psi\>.$$
Thus, $\<X \cdot\psi, \psi\> =0$.
\end{proof}

\begin{prop}\label{no_gen_II_1}
 Let $(M^n, g)$ be a complete connected Riemannian $\spinc$ manifold of dimension $n \geq 3$. Then, every 
 imaginary generalized Killing spinor of type \rm{II} is already an imaginary Killing spinor.
\end{prop}

\begin{proof} We prove the claim by contradiction and assume that there is a Killing spinor $\psi$ to a non-constant Killing function $\i b$, $b\in C^\infty(M, \RR)$. Then, there is a point $x\in M$ where $\nabla b$ is non-zero. In the following, we will always identify $\nabla b$ and $\d b$ using the metric $g$. Then, $\d b$ is non-zero in a neighbourhood $U$ of $x$, and one can find a local orthonormal frame $(e_1,\ldots, e_{n-1}, \frac{\d b}{|\d b|})$ of $TU$. Then, by Lemma~\ref{aux_lem} $\< e_i\cdot \psi, \psi \>=0$ for all $1\leq i\leq n-1$ which will used in following without any further comment. In particular, this implies that the conformal Killing field $V$ (cf. \eqref{def_V}) is parallel to $\text{d}b$ and $\Vert V\Vert=g\left(V,\frac{\d b}{|\d b|}\right)=-\i \left\< \frac{\d b}{|\d b|}\cdot \psi, \psi\right\>$. By Theorem~\ref{rada_thm} $V$ has at most two zeros. Hence, there is an $y\in U$ where $\< \d b\cdot \psi, \psi\>\neq 0$. The following calculations will be carried out at this point $y$.

Take now $e_i$ Clifford multiplied with the Lichnerowicz identity in Lemma~\ref{gen}\eqref{Lich_compl} and its scalar product with $\psi$:
\begin{align*}
 -n(n-1)b^2 \< e_i\cdot \psi, \psi\>  - \i(n-1) \left\<e_i \cdot \d b\cdot \psi, \psi\right\> = \frac{S}{4} \< e_i\cdot \psi, \psi\> + \frac{\i}{2} \<e_i \cdot \Omega\cdot \psi,\psi\>.
\end{align*}

Taking the real part gives
\begin{align}\label{1_Omega}
\i (n-1) \left\<e_i \cdot \d b\cdot \psi, \psi\right\> = \frac{\i}{2} \Omega\left(e_i, \frac{\d b}{|\d b|}\right) \left\<\frac{\d b}{|\d b|} \cdot \psi,\psi\right\>.
\end{align}

On the other hand taking the scalar product of the Ricci identity in Lemma~\ref{gen}\eqref{ric_kill} for $X=e_i$ with $\psi$  and using $\<e_j\cdot \psi, \psi\>=0$ gives
 \begin{align*}
  \frac{1}{2}& \Ric\left(e_i,\frac{\d b}{|\d b|}\right) \left\<\frac{\d b}{|\d b|}\cdot \psi, \psi\right\> - \frac{\i}{2} \Omega\left(e_i, \frac{\d b}{|\d b|}\right) \left\< \frac{\d b}{|\d b|}\cdot \psi, \psi\right\>
= \i \left\< \d b\cdot e_i\cdot \psi, \psi\right\>.
\end{align*}
 
From the imaginary part of this equation  we obtain
 \begin{align}\label{eq_Ric_kill}
  \Ric\left(e_i,\frac{\d b}{|\d b|}\right) \left\<\frac{\d b}{|\d b|}\cdot \psi, \psi\right\> =0 \text{\ and,\ hence,\ }  \Ric\left(e_i,\frac{\d b}{|\d b|}\right) =0,
\end{align}

and the real part gives 
 \begin{eqnarray}\label{omega-db}
- \frac{\i}{2} \Omega\left(e_i, \frac{\d b}{|\d b|}\right) \left\< \frac{\d b}{|\d b|}\cdot \psi, \psi\right\>= \i \left\< \d b\cdot e_i\cdot \psi, \psi\right\>.
\end{eqnarray}

Since $n\geq 3$, \eqref{1_Omega} and \eqref{omega-db} imply 
 \begin{align}\label{ric_kill_2}
\left\< \d b\cdot e_i\cdot \psi, \psi\right\> =0 \quad\text{and} \quad \Omega\left(e_i, \frac{\d b}{|\d b|}\right) =0.
\end{align}

The Ricci identity in Lemma~\ref{gen}\eqref{ric_kill} for $X=\frac{\d b}{|\d b|}$ together with \eqref{ric_kill_2} and \eqref{eq_Ric_kill} gives 

 \begin{align*}
  \frac{1}{2} \Ric\left(\frac{\d b}{|\d b|}, \frac{\d b}{|\d b|}\right) \frac{\d b}{|\d b|} \cdot \psi= -\i |\d b| \psi + n \i |\d b|\psi -2(n-1)b^2 \frac{\d b}{|\d b|} \cdot \psi.
\end{align*}

In particular, $\frac{\d b}{|\d b|}\cdot \psi$ is parallel to $\psi$ and 
$\Vert V \Vert\geq |g(V, \frac{\d b}{|\d b|})|=\left|\< \frac{\d b}{|\d b|}\cdot \psi, \psi\>\right|= |\psi|^2$. Hence, $q_\psi\leq 0$ which gives the contradiction.
\end{proof}
We still have to carry out the $2$-dimensional case.

\begin{thm}\label{gen_II_2}
In dimension $2$, there exists imaginary generalized Killing $\Spinc$ spinors of type \rm{II} with non-constant Killing function.
\end{thm}

\begin{proof} The proof is inspired by the construction of real generalized Killing $\Spinc$ spinors in dimension $2$, cf. \cite[Theorem 2.5]{8}. We consider the two-dimensional Euclidean space $(\mathbb R^2, g_E=\d x^2+\d y^2)$. Then $\{ \partial_x, \partial_y\}$ forms an orthonormal frame. We endow $\mathbb R^2$ with a  conformal metric $\widetilde g$ on $\mathbb R^2$ by requiring the frame $\{\wi \partial_x:=a \partial_x, \wi\partial_y :=a\partial_y\}$ be orthonormal.  Let $\widetilde \nabla$ be the covariant derivative corresponding to $\widetilde g$. The function $a$  will be specified later but depends only on $x$. Then, $ [\wi  \partial_x,\wi   \partial_y] = a' \wi  \partial_y$.
We denote by $\widetilde \nabla$ the Levi-Civita connection on $(\mathbb R^2, \widetilde g)$. 
Using the Koszul formula, one can check that
$$ \wi\nabla_{\wi  \partial_x} \wi  \partial_x =  0\quad \text{and}\quad  \wi\nabla_{\wi  \partial_x} \wi  \partial_y = -a' \wi  \partial_y.$$ 
We denote by $\Sigma \mathbb R^2$ (resp. $\wi\Sigma \mathbb R^2$) the spinor bundle of $(\mathbb R^2, g_E)$ (resp. $(\mathbb R^2, \wi g)$). By a slight abuse of notation, we denote the Clifford multiplication of $(\mathbb R^2, g)$ and $(\mathbb R^2, \wi g)$ by the same symbol $``\cdot"$.
Now, we consider the linear isomorphism of the tangent spaces of $\mathbb{R}^2$ w.r.t. the metrics $g_E$ and $\tilde{g}$ defined  by $ \partial_x\mapsto \wi  \partial_x$ and $\partial_y\mapsto  \wi  \partial_y$. This map lifts to a fibrewise isometric isomorphism of the spinor bundles $\Sigma \mathbb{R}^2\to \wi \Sigma \mathbb{R}^2$, see  \cite{BG}. Using this identification, let $\tilde{\phi}_+$ denote the image  of a positive parallel spinor $\phi_+$ in $\Sigma \mathbb{R}^2$ with $|\phi_+|=1$. Note that $\{\wi\phi_+, \wi\partial_x\cdot \tilde{\phi}_+\}$ forms an orthonormal basis of $\tilde{\Sigma} \mathbb{R}^2$. Let $\phi_-:=\partial_x\cdot \phi_+$. Since $\i \partial_x \cdot \partial_y \cdot \phi_+ = \phi_+$, see \eqref{vol_form}, we have $\partial_y\cdot \phi_{\pm} = \i  \phi_{\mp}$. Using again the identification of the spinor bundles, we get  $\tilde{\phi}_-=\wi \partial_x \cdot \tilde  \phi_+$ and $\wi \partial_y\cdot \tilde{\phi}_{\pm} = \i \tilde{\phi}_{\mp}$. Together with \eqref{nnabla}, we then 
get
\begin{eqnarray*}
\widetilde \nabla_{\widetilde  \partial_x} \tilde{\phi}_\pm = \frac 12 \widetilde g(\widetilde \nabla_{\widetilde  \partial_x} \widetilde  \partial_x,\widetilde  \partial_y) \widetilde  \partial_x \cdot \widetilde  \partial_y \cdot \tilde{\phi}_\pm = 0,
\end{eqnarray*}
and 
\begin{eqnarray*}
 \widetilde \nabla_{\widetilde  \partial_y} \tilde{\phi}_\pm&=& \frac 12 \widetilde g(\widetilde \nabla_{\widetilde  \partial_y} \widetilde  \partial_x,\widetilde  \partial_x) \widetilde  \partial_x \cdot \widetilde  \partial_y\cdot \tilde{\phi}_\pm
= -\frac{a'}{2} \widetilde  \partial_x \cdot \widetilde  \partial_y \cdot \tilde{\phi}_\pm
= \pm \i \frac{a'}{2} \tilde{\phi}_{\pm}.
\end{eqnarray*}

For the Killing spinor on $(\mathbb R^2, \widetilde g)$ we make the following ansatz  
$$\phi = - \cosh (c(x)) \widetilde \phi_+ + \i\sinh (c(x)) \widetilde \phi_-$$
where $c(x, y) = c(x)$, a real function depending only on $x$, will again be specified later.
We calculate 
\begin{align*}\wi \nabla_{\widetilde  \partial_x} \phi =& -c'(x)\sinh (c(x)) \widetilde \phi_+ + \i c'(x) \cosh (c(x)) \widetilde \phi_- \\
  =& \i c' (x) ( \cosh (c(x)) \widetilde  \partial_x \cdot \widetilde \phi_+ + \i\sinh (c(x)) \widetilde  \partial_x \cdot\widetilde \phi_-)\\
 =& -\i c' (x) \widetilde  \partial_x \cdot \phi. 
\end{align*}

Moreover,
\begin{eqnarray*}
\widetilde \nabla_{\widetilde  \partial_y} \phi =  
  - \i \frac{a'}{2} \cosh (c(x)) \tilde{\phi}_+   +\frac{a'}{2} \sinh (c(x)) \tilde{\phi}_- .
\end{eqnarray*}

We now consider the trivial line bundle $L$ on $\mathbb R^2$ with connection form given by an imaginary $1$-form $\i \widetilde \alpha$ satisfying $\widetilde \alpha (\widetilde  \partial_x) = 0$ and $\widetilde \alpha (\widetilde  \partial_y) = \alpha$. Here $\alpha$ is a real function  depending only on $x$. We twist $\widetilde \Sigma \mathbb R^2$ with $L$
which yields a $\Spinc$ structure on $\mathbb R^2$. Let $\sigma$ be a non-zero constant section  of $L$ and consider $\phi \otimes \sigma$. W.l.o.g. let $|\sigma|=1$. On $\wi \Sigma \mathbb R^2 \otimes L$, we consider the twisted connection $\hat { \nabla} = \wi \nabla \otimes \mathrm{Id}  + \mathrm{Id} \otimes \nabla^L$, where $\nabla^L$ is the covariant derivative on $L$ given by $\nabla^L_. \sigma = \i \wi \alpha (.) \sigma$. Then
\begin{align*}\hat \nabla_{\widetilde  \partial_x} (\phi \otimes \sigma) =& -\i c' \widetilde  \partial_x \cdot (\phi \otimes \sigma)\\
\hat \nabla_{\widetilde  \partial_y} (\phi \otimes \sigma) =& - \i\left( \frac{a^{'}}{2} + \alpha\right) \cosh (c(x)) \widetilde \phi_+ \otimes \sigma + \left(\frac {a^{'}}{2} - \alpha\right)\sinh (c(x))  \widetilde \phi_- \otimes \sigma.
\end{align*}
In order to show that  $\phi\otimes \sigma$ is a Killing spinor, we want the last term to be equal to
\begin{eqnarray*}
-\i c' \widetilde  \partial_y \cdot (\phi\otimes \sigma) =    \i c'\sinh (c(x)) \widetilde \phi_+ \otimes \sigma   -c' \cosh (c(x)) \widetilde \phi_- \otimes \sigma.
\end{eqnarray*}
Thus, we should solve
\begin{eqnarray*}
\left\{
\begin{array}{l}
(\frac{a'}{2}  + \alpha) \cosh (c(x)) = -c'\sinh (c(x)),\\
(\frac{a'}{2}  - \alpha)\sinh (c(x)) = -c' \cosh (c(x)).
\end{array}
\right.
\end{eqnarray*}
Any smooth function $c(x)$ with no zeros  together with
\begin{align*} 
 \alpha(x)&= \frac{1}{2}c' \left(\coth (c(x))  - \tanh (c(x))\right),\\
 a'(x)&= -c' \left(\coth (c(x))  + \tanh (c(x))\right),
\end{align*}
such that $a'$ is bounded, gives a solution. E.g. take $c(x)=1+\frac{1}{1+x^2}$. Hence, such a $c$ determines a Killing spinor to the Killing function $-\i c'$.  Note that since $a'$ is required to be bounded, the conformal factor $a$ can be chosen such that it is everywhere positive as requested.

It remains to show that such spinors are of type \rm{II}, i.e., that $q_{\phi\otimes \sigma}$  is positive. By definition $q_{\phi\otimes \sigma} = f^2 - \Vert V \Vert^2_{\widetilde g}$. 
First, note that $f = \vert \phi \otimes \sigma \vert^2_{\widetilde g} = \cosh^2 (c(x)) + \sinh ^2(c(x))$. Moreover, we have $$\Vert V\Vert_{\tilde{g}}^2 = \widetilde g (V, \widetilde \partial_x)^2 +   \widetilde g (V, \widetilde \partial_y)^2 = \<\i \widetilde \partial_x\cdot(\phi \otimes \sigma), \phi \otimes \sigma\>^2 + \<\i \widetilde \partial_y\cdot (\phi \otimes \sigma), \phi\otimes \sigma\>^2.$$
Together with
\begin{eqnarray*}
 &\ & \<\i \widetilde \partial_x\cdot(\phi \otimes \sigma), \phi\otimes \sigma\> \\ &=& \<-\i \cosh (c(x)) \widetilde \psi_- \otimes \sigma  +\sinh (c(x))  \widetilde \psi_+ \otimes\sigma,  -\cosh (c(x)) \widetilde \psi_+\otimes \sigma + \i\sinh (c(x)) \widetilde \psi_-\otimes \sigma \> \\ 
&=& -\sinh (c(x)) \cosh (c(x)) -\sinh (c(x)) \cosh (c(x))  = -2\sinh (c(x)) \cosh (c(x)) 
\end{eqnarray*}
and
\begin{eqnarray*}
&\ &\<\i \widetilde \partial_y\cdot(\phi\otimes \sigma), \phi\otimes \sigma\>\\ &=& \<\cosh (c(x)) \widetilde \psi_-  \otimes \sigma -\i\sinh (c(x)) \widetilde \psi_+ \otimes \sigma, - \cosh (c(x)) \widetilde \psi_+ \otimes \sigma+ \i\sinh (c(x)) \widetilde \psi_- \otimes \sigma\> \\ 
&=& -\i\sinh (c(x)) \cosh (c(x)) + \i\sinh (c(x)) \cosh (c(x)) = 0. 
\end{eqnarray*}
we obtain 
\begin{eqnarray*}
f^2 - \Vert V\Vert^2 &=& (\cosh^2 (c(x)) + \sinh^2(c(x)))^2- 4 \cosh^2 (c(x)) \sinh ^2(c(x))\\
&=& (\cosh^2 (c(x)) - \sinh ^2(c(x)) )^2=1.
\end{eqnarray*}

\end{proof}

{\bf Acknowledgment.} We are indebted to the Institute of Mathematics of
the University of Potsdam, especially to the group of Christian B\"ar, for their hospitality and support. The second author thanks also the Institute of Mathematics  of the University of Leipzig and gratefully acknowledges the financial support  of the Berlin Mathematical school. 

\bibliographystyle{acm}
\bibliography{imag}

\end{document}